\documentclass[a4paper, 11pt]{article}

\usepackage[T1]{fontenc}
\usepackage[utf8]{inputenc}
\usepackage[english]{babel}	
\usepackage{lmodern}

\usepackage{amssymb}
\usepackage{amsthm}
\usepackage{amsmath}
\usepackage{mathtools}
\usepackage{mathrsfs}

\usepackage{verbatim}
\usepackage[shortlabels]{enumitem}
\usepackage{hyperref}
\usepackage[dvipsnames]{xcolor}

\usepackage[autostyle]{csquotes}
\usepackage[style=alphabetic,doi=false,url=false,%
	isbn=false,%
  giveninits=true, backend=biber]{biblatex}
    
\usepackage[capitalise, noabbrev]{cleveref}
	\theoremstyle{plain}
		\newtheorem{theorem}{Theorem}
		\newtheorem{corollary}[theorem]{Corollary}
		\newtheorem{lemma}[theorem]{Lemma}
		\newtheorem{proposition}[theorem]{Proposition}
	\theoremstyle{definition}
		\newtheorem{definition}[theorem]{Definition}
		\newtheorem{remark}[theorem]{Remark}

\DeclareMathOperator{\spt}{spt}
\DeclareMathOperator{\ImP}{Im}
\DeclareMathOperator{\Id}{Id}

\DeclarePairedDelimiter{\abs}{\lvert}{\rvert}

\DeclarePairedDelimiter{\pths}{(}{)}
\DeclarePairedDelimiter{\bkts}{[}{]}
\DeclarePairedDelimiter{\brcs}{\lbrace}{\rbrace}
\DeclarePairedDelimiter{\ango}{\langle}{\rangle}

\newcommand{\stset}{\;\middle|\;}

\newcommand{\numberset}{\mathbb}
\newcommand{\N}{\numberset{N}}
\newcommand{\C}{\numberset{C}}
\newcommand{\Q}{\numberset{Q}}
\newcommand{\R}{\numberset{R}}
\newcommand{\Z}{\numberset{Z}}

\newcommand{\heis}{{\mathbb{H}}} 		

\newcommand{\dc}{d_{\mathrm{c}}}		
\newcommand{\dg}{d_{\mathrm{g}}}		
\newcommand{\leb}{\mathscr{L}}					
\newcommand{\prob}{\mathcal{P}}				
\newcommand\restr[2]{{
  \left.\kern-\nulldelimiterspace 
  #1 
  \vphantom{\big|} 
  \right|_{#2} 
  }}
\DeclareUnicodeCharacter{0229}{\c{e}}

\newcommand{\ie}{\emph{i.e.}}

\def\MySet#1{\expandafter\def\csname mc#1\endcsname{\mathcal{#1}}%
			 \expandafter\def\csname ms#1\endcsname{\mathscr{#1}}%
			 }
\def\ALLvec#1{\ifx#1\ALLvec\else\MySet#1\expandafter\ALLvec\fi}
\ALLvec ABCDEFGHIJKLMNOPQRSTUVWXYZ \ALLvec
\def\MySet#1{\expandafter\def\csname #1#1\endcsname{\mathbf{#1}}%
			 }
\def\ALLvec#1{\ifx#1\ALLvec\else\MySet#1\expandafter\ALLvec\fi}
\ALLvec xuK \ALLvec

\title{Multi-marginal optimal transport on the Heisenberg group\thanks{BP is pleased
		to acknowledge support from Natural Sciences and Engineering Research
		Council of Canada Grant 04658-2018. 
		AP is partially supported by Gruppo Nazionale per l'Analisi Matematica, la Probabilità e le loro Applicazioni (GNAMPA) of the Istituto Nazionale di Alta Matematica (INDAM).}}

\author{Brendan Pass\thanks{
			Department of Mathematical and Statistical Sciences, University of Alberta,
			Edmonton, Alberta, Canada.
			Email: \href{mailto:pass@ualberta.ca}{pass@ualberta.ca}.},
		Andrea Pinamonti\thanks{
			Dipartimento di Matematica, Università degli Studi di Trento, Trento, Italy.
			Email: \href{mailto:andrea.pinamonti@unitn.it}{andrea.pinamonti@unitn.it}.},
		Mattia Vedovato\thanks{
			Dipartimento di Matematica, Università degli Studi di Trento, Trento, Italy.
			Email: \href{mailto:mattia.vedovato@unitn.it}{mattia.vedovato@unitn.it}.}}
\date{\today}

\begin{document}
\maketitle
\begin{abstract}
We consider the multi-marginal optimal transport of aligning several compactly supported marginals on the Heisenberg group to minimize the total cost, which we take to be the sum of the squared Carnot-Carathéodory distances from the marginal points to their barycenter.  Under certain technical hypotheses, we prove existence and uniqueness of optimal maps.  We also point out several related open questions.
\end{abstract}
\section{Introduction}
	
Given Borel probability measures $\mu_1,\mu_2,...,\mu_m$ on a metric space $X$, and a cost function $c:X^m \rightarrow \mathbb{R}$, Monge's multi-marginal optimal transport problem is to minimize

\begin{equation}
\label{equation: general Monge} \tag{MP}
\int_{X} c(x_1,T_2(x_1)...,T_m(x_1)) d\mu_1(x_1)
\end{equation}
among $(m-1)$-tuples of mappings  $(T_2,...,T_m)$  where each $T_i:X \rightarrow X$ pushes $\mu_1$ forward to  $\mu_i$, $(T_i)_\#\mu_1=\mu_i$; that is, $\mu_1(T_i^{-1}(A)) =\mu_i(A)$ for each  measurable $A \subset X$. When $m=2$, \eqref{equation: general Monge} reduces to the well known classical optimal transport problem of Monge, which has many applications both within and outside of mathematics; see \cite{Villani2003,Villani2009,Santambrogio2015} for comprehensive surveys. Of particular interest  are cost functions reflecting the underlying geometry of $X$; the most thoroughly studied (and most important in applications) case arises when $c(x_1,x_2) = d^2(x_1,x_2)$ is the metric distance squared.  For this cost, when $X=\mathbb{R}^n$, a seminal theorem of Brenier \cite{Brenier1987, Brenier1991} asserts that there exists a unique minimizer to \eqref{equation: general Monge}.  This result has been extended to much more exotic geometrical settings, beginning with the work of McCann when $X$ is a Riemannian manifold \cite{McCann2001}.  Of special interest in the present paper will be the case where $X =\heis^n$ is the Heisenberg group equipped with the Carnot-Carathéodory distance, for which the existence and uniqueness of optimal maps was established by Ambrosio and Rigot in \cite{Ambrosio2004}.  Although it is not directly relevant here, we mention that these properties have been extended to a wide class of general cost functions (those satisfying the \textit{twist} condition, injectivity of $x_2 \mapsto D_{x_1}c(x_1,x_2)$; see, for example, \cite{Santambrogio2015}). Moreover, it's worth noting that analogous results are available in more general subriemannian spaces (again taking the squared Carnot-Carathéodory distance $\dc^2$ as a cost function, see \cite{FigalliRifford2010}), and existence holds on the Heisenberg group with the distance cost function ($c=\dc$, see \cite{DePascale2011}).

Multi-marginal problems (ie, \eqref{equation: general Monge} with $m\geq 3$) have received increasing attention in recent years, due to their own fairly wide variety of applications.  In particular, we note that for a natural extension of the quadratic cost on $\mathbb{R}^n$, the existence of unique solutions was proven in a pioneering paper by Gangbo and Swiech \cite{Gangbo1998}.  These results were extended to Riemannian manifolds by Kim and Pass \cite{Kim2015}, but have so far not been established on many other spaces (an exception is Alexandrov spaces; see \cite{Jiang2017}).  More generally, the multi-marginal problem is much more delicate than its two marginal counterpart; existence and uniqueness results have only been established for very special general cost functions, and in fact many examples of multi-marginal cost functions have been exhibited for which solutions to the relaxed, Kantorovich version of the problem (\eqref{equation: general kantorovich} below) do not induce solutions to \eqref{equation: general Monge} and are not unique (see \cite{Pass2015} for an overview).  The purpose of the present short paper is to show that the techniques of Ambrosio-Rigot can be combined with those of Kim-Pass to yield the existence of unique solutions to the multi-marginal optimal transport problem on the Heisenberg group; see Theorem \ref{theorem: monge} and Corollary \ref{corollary: final result} below.

Our proof requires two technical assumptions, not present in either \cite{Ambrosio2004} or \cite{Kim2015}, which may seem surprising to experts.  First, we require in Theorem \ref{theorem: monge} that the mapping from optimally coupled points to their barycenter is injective; in Corollary \ref{corollary: final result} this follows from the sufficient condition that each $\mu_i$ assigns mass $0$ to the set of barycenters of points in the supports on the marginals.  This fact can be proven in the Riemannian setting, but we were unable to prove it on the Heisenberg group without additional assumptions, essentially because semi-concavity of the distance squared fails along the diagonal.  In the Riemannian case, global semi-concavity allows one to use a now quite standard argument, originally introduced in \cite{McCann2001}, to show that differentiability of Kantorovich potentials implies differentiability of the squared distance for optimally coupled points.  The second additional assumption (present in Corollary \ref{corollary: final result} but not Theorem \ref{theorem: monge}) is that \textit{all}, rather than only the first, marginals must be absolutely continuous with respect to Lebesgue measure. This is because  we use almost everywhere differentiability of each Kantorovich potential to ensure differentability of the distance squared from each marginal  point $x_i$ to the barycenter $y$ (under the assumption above ensuring $y \neq x_i$), allowing us to invert the relationship $(x_1,...,x_m) \mapsto y$.  We do not know at this time whether either of these assumptions can be removed.  However, we note that somewhat surprisingly, the semi-concavity of the cost function issue vanishes when we replace the quadratic cost with the $p$ power-cost for $p>2$; see \cref{subsection: higher p}.

The manuscript is organized as follows: in section 2, we recall basic facts about both the Heisenberg group and multi-marginal optimal transport.  In section 3, we introduce the cost function we will work with and establish some preliminary properties of it.  Section 4 is devoted to the precise statement and proof of our main results: Theorem \ref{theorem: monge} and Corollary \ref{corollary: final result}.

\section{Preliminaries: multimarginal optimal transport on \texorpdfstring{$\heis^n$}{Hn}}

\subsection{Heisenberg group}
In this paper, we will represent the Heisenberg group $\heis^{n}$ as the set $\C^n\times\R \equiv \R^{2n+1}$; thus its points will be described as $x=[z,t]=[\zeta+i\eta,t]=(\zeta,\eta,t)$, with $z\in\mathbb C^n$, $\zeta,\eta\in\R^n$, $t\in\R$. The group operation on $\heis^{n}$ will be defined as follows: whenever $x=[z,t]\in \heis^{n}$ and $x'=[z',t']\in \heis^{n}$,
	\begin{equation}\label{equation: group operation}
	x\cdot x'	\doteq \bkts*{z+z', t+t'+ 2 \ImP \pths*{\ango*{ z,\bar{z'}}}}.
	\end{equation}
As a consequence, it is easy to verify that the group identity is the origin 0 and the inverse of a point is given by $[z,t]^{-1}=[-z,-t]$.
Even though they will only play a role ``in the background'' in the current work, we also introduce for the sake of completeness the following family of \emph{non-isotropic dilations}, since they are fundamental in the geometry of the Heisenberg group: if $x=[z,t]\in \heis^{n}$  and $\lambda>0$,
	\begin{equation}\label{equation: dilations}
	\delta_{\lambda}(x)\doteq[\lambda z,\lambda^2t]. 
	\end{equation}
The Heisenberg group $\heis^n$ admits the structure of a Lie group of topological dimension $2n+1$. Its Lie algebra $\mathfrak{h}_n$ of left invariant vector fields is (linearly) generated by
\begin{equation}\label{equation: left inv vector fields}
X_j =\frac{\partial}{\partial x_j}+2y_j\frac{\partial}{\partial t},\quad 
Y_j=\frac{\partial}{\partial y_j}-2x_j\frac{\partial}{\partial t },\quad 
\text{for $j=1,\dots,n$;}
\quad\text{and } Z=\frac{\partial}{ \partial t};
\end{equation}
notice that the only non-trivial commutator relation is
\begin{equation}
[X_j,Y_j] = - 4Z,
\end{equation}
valid for any $j=1,\dots,n$. The vector fields $X_1,\dots,X_{n}, Y_1,\dots, Y_n$ will be called \emph{horizontal vector fields}: the group $\heis^{n}$ endowed with this Lie algebra has the structure of a Carnot group. 

In our main Monge-type result, we will need the following geometric result about the structure of the Heisenberg group: the space $\heis^n$ can be parametrized by a system of ``adapted spherical coordinates''.

\begin{proposition}[Spherical coordinates, {\cite[Proposition 3.9]{Ambrosio2004}}]\label{proposition: parametrization}
Let 
	\begin{equation}
		\mathbb{S}\doteq \brcs*{a+i b \in\C^n\stset \abs{a+ib}=1},
	\end{equation}	
and let $\mathbf{D}\doteq \mathbb{S}\times(-2\pi,2\pi)\times (0,+\infty).$
Define the following map:
	\begin{equation}
	\begin{aligned}
		\Upsilon: \qquad	\mathbf{D} &\longrightarrow 		\heis^n\\
					(a+ib,v,r) &\longmapsto	    (\xi_1,\dots,\xi_n,\eta_1,\dots,\eta_n,t) 
	\end{aligned}
	\end{equation}
with 
	\begin{equation}
	\begin{aligned}
	\xi_j 	&\doteq \frac{b_j(1-\cos v)+a_j \sin v}{v} r\\
	\eta_j 	&\doteq \frac{-a_j(1-\cos v)+b_j \sin v}{v} r\\
	t		&\doteq 2\frac{v-\sin v}{v^2} r^2.
	\end{aligned}
	\end{equation}
Then $\Upsilon$ is a $C^1$ diffeomorphism from $\mathbf{D}$ onto $\heis^n\setminus L$, where $L=\{0\} \times \mathbb{R}$ is the vertical axis.
\end{proposition}

Thanks to \cref{proposition: parametrization}, one can define a generalization of the exponential map to the Heisenberg group:

\begin{definition}[Exponential map on $\heis^n$]\label{definition: exponential}
Let $\Upsilon$ be the diffeomorphism introduced in \cref{proposition: parametrization}. We define the \emph{exponential map} $\exp_\heis: \C^n\times \bkts{-\frac{\pi}{2},\frac{\pi}{2}}\to \heis^n$ as
	\begin{equation}
	\exp_\heis (A+i B, w)\doteq \Upsilon\pths*{ \frac{A+iB}{\abs{A+i B}}, 4 w , \abs{A+iB}}
	\end{equation}
if $A+iB\neq 0$, and $\exp_\heis(0,w)=0$ for all $w$.
\end{definition}

As shown in \cite[Lemma 6.8]{Ambrosio2004}, this definition is consistent with the classical Riemannian one: one can define in a standard way a Riemannian approximation to the subriemannian structure of $\heis^n$; and now $\exp_\heis$ can be recovered as a limit of Riemannian exponential coordinates.

Other important properties of the Heisenberg group (such as the metric structure induced by the Carnot-Carathéodory distance) will be introduced in \cref{section: cost functions}; we refer the reader to \cite{SerraCassano2016} for a more complete introduction to the Heisenberg group and Carnot groups in general.
\subsection{Multimarginal optimal transport}
In what follows, $\mu_1,\dots,\mu_m\in\prob(\heis^n)$ will be Borel probability measures on the Heisenberg group $\heis^n$. We'll also assume they are absolutely continuous with respect to the Lebesgue measure $\leb^{2n+1}$ (which is, up to a constant, the Haar measure of $\heis^n$).
We will call \textbf{Kantorovič problem} the following minimization problem:
	\begin{equation}
	\label{equation: general kantorovich} \tag{KP}
	\inf \brcs*{\int_{(\heis^n)^m} c(x_1,\dots,x_m) \,d\gamma(x_1,\dots,x_m)\stset \gamma \in\Pi\pths*{\mu_1,\dots,\mu_m}}
	\end{equation}
where $c:(\heis^n)^m\to\R$ is a suitable cost function, $\Pi\pths*{\mu_1,\dots,\mu_m}$ is the following family of \textbf{transport plans}:
	\begin{equation} \label{equation: transport plans}
	\Pi\pths*{\mu_1,\dots,\mu_m}\doteq \brcs*{\mu \in\prob\pths*{(\heis^n)^m} \stset \text{${\pi_i}_\sharp\mu = \mu_i$ for all $i=1,\dots,m$}},
	\end{equation}
and ${\pi_i} _\sharp\mu$ denotes the push-forward measure of $\mu$ with respect to the projection on the $i^{\textrm{th}}$ component. 
\par
We also consider the following \textbf{dual problem}:
	\begin{equation}
	\label{equation: general dual problem} \tag{DP}
		\sup\brcs*{\sum_{i=1}^m \int_{\heis^n} u_i(x_i)\,d\mu_i(x_i)\stset 
					\begin{gathered}
						\text{$u_i\in L^1_{\mu_i}(\heis^n)$ for all $i=1,\dots,m$,}\\
						{\textstyle \sum_{i=1}^m u_i(x_i)\leq c(x_1,\dots,x_m)}				
					\end{gathered}
					}.
	\end{equation}
In what follows, we will often denote by $\xx$ the $m$-tuple $(x_1,\dots,x_m)\in(\heis^n)^m$.

\paragraph{Restriction to compact subsets} From now on, we will always assume $\mu_1,\dots,\mu_m$ are compactly supported. If we define the compact sets
\begin{equation}\label{equation bigK}
K_i\doteq \spt(\mu_i)\subset \heis^n \quad \text{and} \quad \KK\doteq \prod_i\spt(\mu_i) \subset (\heis^n)^m,
\end{equation}
then the Kantorovič problem~\eqref{equation: general kantorovich} is actually equivalent to 
	\begin{equation}
	\label{equation: kantorovich} \tag{$\mathrm{KP}_c$}
	\inf \brcs*{\int_{\KK} c(x_1,\dots,x_m) \,d\gamma(x_1,\dots,x_m)\stset \gamma \in\Pi\pths*{\mu_1,\dots,\mu_m}}.
	\end{equation}

Indeed, for any $\mu\in\Pi\pths*{\mu_1,\dots,\mu_m}$, we have $\pi_i(\spt \mu)\subset \spt(\pi_i\sharp \mu)=\spt(\mu_i)$ (by elementary properties of measures and by continuity of $\pi_i$), thus any transport plan is supported in $\prod_i\spt(\mu_i)=\KK$. The dual problem~\eqref{equation: general dual problem}, on the other hand, is actually equivalent to the following:
	\begin{equation}
	\label{equation: dual problem} \tag{$\mathrm{DP}_c$}
		\sup\brcs*{\sum_{i=1}^m \int_{K_i} u_i(x_i)\,d\mu_i(x_i)\stset 
					\begin{gathered}
						\text{$u_i\in L^1_{\mu_i}(K_i)$ for all $i=1,\dots,m$,}\\
						{\textstyle \sum_{i=1}^m u_i(x_i)\leq c(x_1,\dots,x_m)}				
					\end{gathered}
					},
	\end{equation}	
in the sense that any solution of~\eqref{equation: general dual problem} restricts to a solution of~\eqref{equation: dual problem} and vice versa, a solution of~\eqref{equation: dual problem} extends to a solution of~\eqref{equation: general dual problem}, for example by defining the value of $u_i$ to be $-\infty$ out of $K_i$.

\begin{definition}
Let $c:(\heis^n)^m\to\R$ be a cost function. We say that the $m$-tuple of functions $\uu=\pths*{u_1,\dots,u_m}:\heis^n \to(\R\cup\brcs{-\infty})^m$ is \emph{$c$-conjugate} if, for some $U\subset\heis^n$ and for every fixed $1\leq i \leq m$ and $x_i\in \heis^n$, it holds:
	\begin{equation}
	u_i(x_i) = \inf\brcs*{c(x_1,\dots,x_m)-\sum_{j\neq i} u_j(x_j)\stset \text{$x_j\in U$ for any $j\neq i$}}.
	\end{equation}
Moreover, when $\uu=\pths*{u_1,\dots,u_m}$ is such an $m$-tuple, we denote by $\Gamma_\uu$ the set
	\begin{equation}\label{equation: graph of superdifferential}
	\Gamma_\uu \doteq \brcs*{\xx\in U^m \stset \sum_{i=1}^m u_i(x_i)=c(\xx)}.
	\end{equation}
\end{definition}

\begin{definition}[Cyclical monotonicity]\label{definition: cyclical monotonicity}
Let $\Gamma\subset (\heis^n)^m$. We say that $\Gamma$ is \textbf{$c$-cyclically monotone} if the following holds: for any $N\in\N$, for any $m$-tuple $(\sigma_1,\dots,\sigma_m)$ of permutations of $N$ elements, and for any subset
	\begin{equation}
 	\brcs*{\xx^j=(x_1^j,\dots , x_m^j)\in (\heis^n)^m \stset j=1,\dots,N}\subset \Gamma
	\end{equation}
the following inequality holds true:
	\begin{equation}
	\sum_{j=1}^N c(x_1^j,\dots , x_m^j)\leq \sum_{j=1}^N c\pths[\Big]{x_1^{\sigma_1(j)},\dots x_m^{\sigma_m(j)}}.
	\end{equation}
\end{definition}

\begin{proposition}
Let $c:(\heis^n)^m\to\R$ be a cost function, and let $(u_1,\dots,u_m):\heis^n\to\R^m$ be a $c$-conjugate $m$-tuple. Then the set $\Gamma_\uu$ defined in \cref{equation: graph of superdifferential} is $c$-cyclically monotone.
\end{proposition}

\begin{proof}
Fix $N\in\N$, a $m$-tuple $\pths{\sigma_1,\dots,\sigma_m}$ of $N$-permutations and a subset $\brcs*{\xx^j}_{j=1}^N\subset\Gamma_\uu$. Then
	\begin{equation}
	\begin{split}
		\sum_{j=1}^N c\pths[\Big]{x_1^{\sigma_1(j)},\dots x_m^{\sigma_m(j)}} &\geq
		\sum_{j=1}^N \sum_{i=1}^m u_i \pths{x_i^{\sigma_i(j)}} = \sum_{i=1}^m \sum_{j=1}^N u_i \pths{x_i^{\sigma_i(j)}} = \\
		&= \sum_{i=1}^m \sum_{j=1}^N u_i (x_i^j) = \sum_{j=1}^N c(x_1^j,\dots x_m^j).
	\end{split} 
	\end{equation}
\end{proof}

The following Existence and Duality Theorem is a classical result in optimal transportation: even though it's  not present in literature at this level of generality (multimarginal and subriemannian), the classical proof can be easily adapted to this context (see for example \cite[Theorem 1.17 and Remark 1.18]{Ambrosio2013}).
\begin{theorem}[Existence and Duality]\label{theorem: existence and duality}
Let $\mu_1,\dots,\mu_m\in\prob(\heis^n)$ be compactly supported, Borel probability measures. Let $c:(\heis^n)^m\to\R$ be a continuous cost function. Then:
	\begin{enumerate}[(i)]
	\item \textnormal{Existence for \eqref{equation: kantorovich}:} there exists $\gamma\in\Pi(\mu_1,\dots,\mu_m)$ which achieves the minimum in the Kantorovič problem~\eqref{equation: kantorovich};
	\item \textnormal{Existence for \eqref{equation: dual problem}:} there exists a $c$-conjugate $m$-tuple $(u_1,\dots,u_m)$ which achieves the maximum in the dual problem~\eqref{equation: dual problem};
	\item \textnormal{Duality:} the minimum in~\eqref{equation: kantorovich} and the maximum in~\eqref{equation: dual problem} coincide;
	\item If $\gamma$ is an optimal plan and $\uu=(u_1,\dots,u_m)$ is a $c$-conjugate solution to~\eqref{equation: dual problem}, then $\spt(\gamma)\subset \Gamma_{\uu}$.
	\end{enumerate}
\end{theorem}

\begin{remark}
If $\uu$ is a $c$-conjugate $m$-tuple which maximizes~\eqref{equation: dual problem}, we can assume
	\begin{equation}
	u_i(x_i)=\inf\brcs*{c(x_1,\dots,x_m)-\sum_{j\neq i} u_j(x_j)\stset \text{$x_j\in K_j$ for any $j\neq i$}}.
	\end{equation}
for $\mu_i$-almost every $x_i\in\heis^n$.
\end{remark}

\section{Cost functions associated to distances}
\label{section: cost functions}
We now specify what kind of cost functions we are interested in: namely, we adapt the idea of ``cost functions which grow quadratically with the distance'' to the multimarginal problem. The same problem was approached in \cite{Gangbo1998} in the Euclidean space $\R^n$ with the standard distance: here the cost $c(\xx)=\sum_{i\neq j} \abs{x_i-x_j}^2$ was considered. When moving to the Riemannian case (see \cite{Kim2015}), a natural way to generalize this is to take the ``barycentric cost'': for $x_i,y\in M$, with $(M,g)$ Riemannian manifold and $d$ distance associated to the metric $g$, one takes the cost
	\begin{equation}
	c(\xx)=\inf_{y} \sum_i d^2 (x_i,y),
	\end{equation}
which in $\R^n$ is actually equivalent to the Gangbo-\'{S}wiȩch one.
The same choice is made here for the Heisenberg group. 

\begin{definition}[Barycentric cost associated to a distance]\label{definition: distance cost function}
For a given distance $d$ on $\heis^n$, we define the cost function 
$c_d:(\heis^n)^m \to \R$
associated to $d$ as the function
	\begin{equation}\label{equation: distance cost function}
		c_d(x_1,\dots,x_m)\doteq \inf_{y\in\heis^n}\brcs*{ \sum_{i=1}^m d^2(x_i, y) }.
	\end{equation}
We'll drop the subscript $d$ when the metric is clear from the context. 
\end{definition}

The infimum in \eqref{equation: distance cost function} is actually a minimum, and the minimum points will play a fundamental role in our work:

\begin{definition}[Barycenters]
We'll say that $y\in\heis^n$ is a \textbf{barycenter} for the $m$-tuple $(x_1,\dots,x_m)$ if $y$ realizes the infimum in \cref{equation: distance cost function}. We denote by $\mathbf{b}:(\heis^n)^m\to\msP(\heis^n)$ the multivalued map which associates an $m$-tuple with its (never empty) set of barycenters. Here $\msP(\heis^n)$ is the power set of $\heis^n$.
\end{definition}

\begin{remark}\label{remark: inf on K}
For any compact set $K\subset(\heis^n)^m$ we can find a larger compact set $K'$ which contains all the barycenters of $m$-tuples $(x_1,\dots, x_m)\in K$: indeed, $c_d$ is upper semi-continuous, thus bounded from above by a value $M>0$ in $K$; hence the set
	\begin{equation}
		H\doteq \bigcup_{\xx\in K}\brcs*{y\in\heis^n \stset \sum_{i=1}^m d^2(x_i,y)\leq M}
	\end{equation}
	contains all the barycenters and is easily seen to be bounded; it suffices to take $K'=\bar{H}$.	
	In particular, if $(x_1,\dots, x_m)\in K^m$, then
	\begin{equation}
		c_d(x_1,\dots,x_m)\doteq \inf_{y\in K'}\brcs*{ \sum_{i=1}^m d^2(x_i, y) }.
	\end{equation}
Notice that with the same argument one can easily prove that the cost function $c_d$ is actually continuous.
\end{remark}

In the following \namecref{lemma: swap}, we show that if $\uu$ is $c_d$-conjugate and $y$ is a barycenter for two different $m$-tuples $\xx, \bar{\xx}$ of $\Gamma_\uu$, then it is also a barycenter for the $m$-tuples obtained by swapping the $i^{\mathrm{th}}$-components of $\xx$ and $\bar{\xx}$. 
\begin{lemma}\label{lemma: swap}
Let:
	\begin{itemize}
	\item $c=c_d:(\heis^n)^m\to\R$ be a cost function associated to a distance $d$ as in \cref{definition: distance cost function};
	\item $\uu=(u_1,\dots,u_m):\heis^n\to\R^m$ be a $c_d$-conjugate $m$-tuple in $U\subset \heis^n$;
	\item $\xx, \bar{\xx}$ belong to the set $\Gamma_\uu$ defined in \cref{equation: graph of superdifferential}.
	\end{itemize}		
	Assume the point $y\in\heis^n$ is a barycenter for both $\xx$ and $\bar{\xx}$. Then $y$ is also a barycenter for the $m$-tuple $\xx'\doteq (\bar{x}_1,x_2,\dots,x_m)$ and for the $m$-tuple $\bar{\xx}'\doteq (x_1,\bar{x}_2,\dots,\bar{x}_m)$.
\end{lemma}

\begin{proof}
By $c$-cyclical monotonicity, 
	\begin{equation}
	\begin{split}
	c(\xx)+c(\bar{\xx})&\leq c(\xx')+c(\bar{\xx}')\\
	&\leq \dc^2(\bar{x}_1,y)+\sum_{i=2}^m \dc^2(x_i,y)+\dc^2(x_1,y)+\sum_{i=2}^m \dc^2(\bar{x}_i,y)\\
	&= \sum_{i=1}^m \dc^2(x_i,y) + \sum_{i=1}^m \dc^2(\bar{x}_i,y)\\
	&= c(\xx)+c(\bar{\xx}).
	\end{split}
	\end{equation}
We must therefore have equality throughout the above string of inequalities, meaning that 
	\begin{equation}\label{eqn: optimality for mixed points}
	c(\xx') = \dc^2(\bar{x}_1,y) +\sum_{i=2}^m \dc^2(x_i,y)\text{ and }c(\bar{\xx}') = \dc^2(x_1,y) +\sum_{i=2}^m \dc^2(\bar{x}_i,y),
	\end{equation}
as desired.
\end{proof}

Our main result concerns the cost function associated to a special distance on $\heis^n$, \ie{} the subriemannian Carnot-Carathéodory distance, which we introduce here:
\begin{definition}[Carnot-Carathéodory distance]\label{definition: CC distance}
We call a \emph{subunit curve} a Lipschitz curve $\gamma:[0,T]\longrightarrow\heis^n$ such that for a.e.\ $t\in[0,T]$,
	\begin{equation}
	\begin{gathered}
		\dot{\gamma}(t)=\sum_{j=1}^n a_j(t)X_j(\gamma(t))+b_j(t)Y_j(\gamma (t))
		\\
		\text{and }\sum_{j=1}^n a_j^2(t)+b_j^2(t) \leq 1
	\end{gathered}	
	\end{equation}
with $a_1,...,a_n,b_1,\dots,b_n$ measurable coefficients. The Carnot-Carathéodory distance between the points $x,y\in\heis^n$ is defined as
	\begin{equation}\label{equation: definition of dcc}
	d_c(x,y)\doteq \inf\brcs*{T\geq 0\stset
								\begin{gathered}
									\textrm{there exists a subunit curve $\gamma:[0,T]\rightarrow\heis^n$}\\
									\textrm{such that $\gamma(0)=x$ and $\gamma(T)=y$}
								\end{gathered}											
							}.
	\end{equation}
\end{definition}

\begin{remark}
The Carnot-Carathéodory distance on $\heis^n$ is an example of a \emph{left invariant} and \emph{homogeneous} metric, \ie{}, whenever $x,\bar{x}\in\heis^n$, $p\in\heis^n$ and $\lambda>0$, the following hold:
	\begin{equation}
	d_{c}(p\cdot x,p\cdot \bar{x})=d_c(x,\bar{x})\quad \text{and}\quad d_c(\delta_\lambda(x),\delta_\lambda(\bar{x}))=d(x,\bar{x}),
	\end{equation}
where $\delta_\lambda([z,t])=[\lambda z,\lambda^2 t]$ for any $[z,t]\in\C^n\times\R\simeq\heis^n$. It's easy to prove that any left invariant and homogeneous metric is equivalent to $d_c$ (see for example \cite[Corollary 5.1.5]{Bonfiglioli2007}). This also means that the notion of ``Lipschitz function'' does not depend on the particular metric chosen, as long as it is left invariant and homogeneous.
\end{remark}

We collect here some properties of the (squared) distance function, which are mostly proved in \cite{Ambrosio2004}; some of them will turn out to be useful later, the others are stated for the sake of completeness.
\begin{proposition}[Differentiability properties of $\dc^2$]
Let us denote, with an abuse of notation, $\dc(z)\doteq \dc(0,z)$ for any $z\in\heis^n$. Let $L=\brcs*{se_{2n+1}\stset s\in\R}$ be the vertical axis. Then:
	\begin{enumerate}[(a)]
	\item If $z\notin L$, then $\dc^2$ is differentiable at $z$.
	\item $X_j\dc^2(0)=0$ and $Y_j\dc^2(0)=0$; moreover, $Z^{\pm}\dc^2(0)=\pm\pi$.
	\item If $z=[0,\dots,0,t]\in L\setminus \brcs{0}$, then $\dc^2(z)=\pi \abs{t}$, thus, $Z\dc^2(z)= \mathrm{sgn}(t) \pi$; however, $\dc^2$ is not differentiable at $z$ along the directions $X_j,Y_j$ with $j=1,\dots,n$.
	\end{enumerate}
\end{proposition}

We now introduce the notion of $d^2$-concavity, which plays a key role in the two-marginals optimal transport theory, and is closely related to the (multi-marginal) notion of $c_d$-conjugacy (see \cref{lemma: dconcave}).

\begin{definition}
Let $u:\heis^n\to \R$. Let $d$ be a metric on $\heis^n$. We say that $u$ is $d^2$-concave if
	\begin{equation}\label{equation: dconcave}
	u(x)=\inf_{y\in U}\brcs*{d^2(x,y)-\phi(y)}\qquad\forall x\in\heis^n
	\end{equation}
for some non-empty set $U\in\heis^n$ and $\phi: U \to \R\cup \brcs{-\infty}$, $\phi \not\equiv -\infty$ (see \cite[Definition 4.1]{Ambrosio2004}).
\end{definition}

\begin{remark}
If $u$ is $d^2$-concave, then \cref{equation: dconcave} also holds with $U=\heis^n$, up to defining $\phi\equiv -\infty$ on $\heis^n\setminus U$. 
\end{remark}

\begin{definition}
Let $u:\heis^n\to \R$ be a function, and let $d$ be a metric on $\heis^n$. We define the $d^2$-superdifferential of $u$ at $x\in\heis^n$ as
	\begin{equation}
	\partial_{d^2}u(x)\doteq \brcs*{y\in\heis^n \stset d^2(x,y)-u(x)\leq d^2(z,y)-u(z)\,\, \forall z\in\heis^n}.
	\end{equation}
We denote by $\partial_{d^2}u$ the graph of the $d^2$-superdifferential:
	\begin{equation}
	\partial_{d^2}u \doteq \brcs*{(x,y)\in\heis^n\times\heis^n\stset y\in\partial_{d^2}u(x)}.
	\end{equation}
\end{definition}

\begin{remark}
If $u$ is $d^2$-concave, then $\partial_{d^2}u$ coincides with the set
	\begin{equation}
	\brcs*{(x,y)\in\heis^n\times\heis^n\stset u(x)+\phi(y)=d^2(x,y)}.
	\end{equation}
\end{remark}

In the following \namecref{lemma: dconcave}, we specify the relation between $c_d$-conjugacy and $d^2$-concavity: in particular, we show that each component of a $c_d$-conjugate solution $\uu$ to the dual problem \eqref{equation: dual problem} is $d^2$-concave.

\begin{lemma}\label{lemma: dconcave}
Let $c_d:(\heis^n)^m\to\R$ be a cost function associated to a homogeneous and left invariant distance $d$, let $\mu_1,\dots,\mu_m$ be compactly supported, absolutely continuous probability measures on $\heis^n$ and let $(u_1,\dots,u_m):\heis^n\to\R^m$ be a $c_d$-conjugate solution to the dual problem~\eqref{equation: dual problem}. Then each $u_j$ is $d^2$-concave. Moreover, if $\bar{\xx}\in\Gamma_\uu$ and $y\in\heis^n$ is a barycenter for $\bar{\xx}$, then $y\in \partial_{d^2}u_i(\bar{x}_i)$ for all $i\in\brcs{1,\dots,m}$. 
\end{lemma}

\begin{proof}
Notice that for all $x_i\in\heis^n$ it holds:
	\begin{equation}
	u_i(x_i)=\inf_{x_j\in K_j, j\neq i}\brcs*{\inf_{y\in K'} \brcs*{\sum_{j=1}^m d^2(x_j,y)} - \sum_{j\neq i}u_j (x_j)},
	\end{equation}
where $K_j$ is defined in \cref{equation bigK} and $K'$ is a suitable compact set (see \cref{remark: inf on K}); thus
	\begin{equation}
	u_i(x_i)=\inf_{y\in K'}\brcs*{d^2(x_i,y)+\inf_{x_j\in K_j, j\neq i}\brcs*{\sum_{j\neq i}(d^2(x_j,y)-u_j (x_j))}}.
	\end{equation}
This proves the first statement. Let now $\bar{\xx}\in\Gamma_\uu$, and $y$ be a barycenter. Then, taking for example $i=1$, we have
	\begin{equation}
	d^2(\bar{x}_1,y)-u_1(\bar{x}_1)=\sum_{j>1}(u_j(\bar{x}_j)-d^2(\bar{x}_j,y));
	\end{equation}
since $\uu$ is $c_d$-conjugate, we then have for any $z\in\heis^n$
	\begin{equation}
	d^2(\bar{x}_1,y)-u_1(\bar{x}_1)\leq c_d(z,\bar{x}_2,\dots, \bar{x}_m)-\sum_{j>1}d^2(\bar{x}_j,y)-u_1(z)\leq d^2(z,y)-u_1(z),
	\end{equation}
	which is what we needed to prove $y\in\partial_{d^2}u_1(\bar{x}_1)$.
\end{proof}

\begin{corollary}\label{corollary: pansu_diff}
Let $c:(\heis^n)^m\to\R$ be the cost function associated to the Carnot-Carathéodory distance $\dc$; let $\mu_1,\dots,\mu_m$ be compactly supported, absolutely continuous probability measures on $\heis^n$ and let $(u_1,\dots,u_m):\heis^n\to\R^m$ be a $c$-conjugate solution to the dual problem~\eqref{equation: dual problem}. Then each $u_j$ is Lipschitz, and thus Pansu differentiable almost everywhere; moreover, $Z u_j$ also exists almost everywhere in $\heis^n$.
\end{corollary}

\begin{proof}
This is a consequence of \cite[Lemma 4.5 and Lemma 4.6]{Ambrosio2004}.
\end{proof}

\begin{remark}
We refer to \cite[Paragraph 2.2]{Ambrosio2004} for a definition of Pansu differentiability and related issues, such as the Rademacher-type Theorem which is needed for \cref{corollary: pansu_diff}. What is important to us is that each $u_j$ is differentiable along $X_1,\dots,X_n$, $Y_1,\dots,Y_n$ and $Z$ for a.e.~$x\in\heis^n$, in the following sense: for any $k=1,\dots,2n+1$, the map $s\mapsto u_j(x\cdot \delta_s(e_k))$ is differentiable at $s=0$.
\end{remark}

Now fix $\dc$ to be the Carnot-Carathéodory distance, and denote by $c$ the associated cost function. Let $\uu=(u_1,\dots,u_m)$ be a $c$-conjugate solution to the dual problem~\eqref{equation: dual problem}, and let now $E_i\subset\heis^n$ be the ($\leb^{2n+1}$-negligible) set where $u_i$ fails to be differentiable along some of the directions $X_j,Y_j,Z$. Define the set 
	\begin{equation}\label{equation: Omega}
	\Omega_\uu\doteq \prod_{i=1}^m(\heis^n\setminus E_i).
	\end{equation}
Then, if $\mu_i \ll \leb^{2n+1}$ for all $i=1,\dots,m$ and $\mu\in\Pi(\mu_1,\dots,\mu_m)$ is a transport plan, one has:
	\begin{equation}
	\mu((\heis^n)^m\setminus\Omega_\uu)\leq \sum_{i=1}^m\mu\pths*{\heis^n\times\dots\times E_i \times\dots\times\heis^n}=\mu_i(E_i)=0,
	\end{equation}
thus $\Omega_\uu$ has full $\mu$-measure in $(\heis^n)^m$.

\section{Main results}

In the following definition, we introduce a condition which will turn out to be useful in our main result: we assume that in a suitable subset of $\heis^n$ the barycenter map can be inverted.
\begin{definition}
Let $c_d:(\heis^n)^m\to\R$ be a cost function associated to a distance $d$, and let $\Gamma \subset (\heis^n)^m$. We say that $\Gamma$ satisfies the condition \eqref{equation:condition_C1} if:
	\begin{equation}\label{equation:condition_C1}\tag{C1}
	\parbox{.8\textwidth}{\centering \itshape
	there exists $\Gamma_1 \subset \Gamma$ such that $\leb^{2n+1}(\pi_1(\Gamma\setminus \Gamma_1))=0$ and the restricted barycenter map $\restr{\mathbf{b}}{\Gamma_1}$ is injective.	
	}
	\end{equation}
Equivalently,
	\begin{equation}\tag{\ref{equation:condition_C1}'}
	\parbox{.8\textwidth}{\centering \itshape
	there exists $\Gamma_1 \subset \Gamma$ such that $\leb^{2n+1}(\pi_1(\Gamma\setminus \Gamma_1))=0$, and whenever $\xx$, $\bar{\xx}\in\Gamma_1$ have a common barycenter $y\in\heis^n$, $\xx$ and $\bar{\xx}$ must coincide.
	}
	\end{equation}
\end{definition}

We are now ready to prove a general Monge-type \namecref{theorem: monge} under the assumption \eqref{equation:condition_C1} on the invertibility of the barycenter map. We will then state a condition on the marginals which ensures \eqref{equation:condition_C1}.

\begin{theorem}\label{theorem: monge}
Let $\mu_1,\dots,\mu_m$ be compactly supported, absolutely continuous probability measures on $\heis^n$, and let $c$ be the cost associated to $\dc$. Let $\uu=(u_1,\dots,u_m)$ be a $c$-conjugate solution to the dual problem~\eqref{equation: dual problem}. Assume that the set $\Gamma_\uu$ defined in \eqref{equation: graph of superdifferential} satisfies the condition \eqref{equation:condition_C1} for some subset $\Gamma_1\subset \Gamma_\uu$. Then:
\begin{enumerate}[(i)]
	\item Any optimal transport plan $\gamma$ which achieves the minimum in \eqref{equation: kantorovich} is induced by a transport map $\Psi$ over the first variable $x_1$.
	\item The optimal transport plan $\gamma$ and map $\Psi$ are both unique.
	\item Let $\exp_\heis$ denote the subriemannian exponential map defined in \cref{definition: exponential}; let 
			\begin{equation}
			\mathbf{b}_{\Gamma_1}^{-1}: \mathbf{b}(\Gamma_1)\to \Gamma_1 \subset (\heis^n)^m
			\end{equation}	
	denote the inverse of the barycenter map, and let ${\mathbf{f}}: \mathbf{b}(\Gamma_1)\to(\heis^n)^{m-1}$ be defined by ${\mathbf{f}}_i=(\mathbf{b}_{\Gamma_1}^{-1})_i$ for $i=2,\dots,m$. Then $\Psi$ can be represented ($\mu_1$-a.e.) as 
		\begin{equation}\label{eqn: optimal maps structure}
		\Psi(x_1)={\mathbf{f}}\circ \psi \, (x_1),
		\end{equation}
	where
		\begin{equation}
		\psi(x_1)={x_1 \cdot \exp_\heis (-X u_1 (x_1)-iY u_1(x_1), -Z u_1(x_1))}
		\end{equation}
\end{enumerate}
\end{theorem}

\begin{proof}
Recall that a transport plan $\mu\in\Pi(\mu_1,\dots,\mu_m)$ is induced by a transport map if and only if there exists a $\mu$-measurable set $\tilde{\Gamma}\subset (\heis^n)^m$ where $\mu$ is concentrated, such that for $\mu_1$-a.e.~$x_1$ there exists only one $(x_2,\dots,x_m) = \Psi(x_1)\in (\heis^n)^{m-1}$ such that $(x_1,x_2,\dots,x_m)\in\tilde{\Gamma}$; in this case $\mu$ is induced by the map $\Psi$ (see \cite[Lemma 1.20]{Ambrosio2013}).

We already know from \cref{theorem: existence and duality} that, if $\gamma$ is an optimal Kantorovič plan, then it is concentrated on the set $\Gamma_\uu$ for some $c$-conjugate solution $\uu=(u_1,...,u_m)$ to the dual problem \eqref{equation: general dual problem}. Let now $\Omega_\uu$ be as defined in \eqref{equation: Omega}; by condition \eqref{equation:condition_C1} and by absolute continuity of $\mu_1$, the projection $\pi_1(\Omega_\uu \cap \Gamma_1)$ has full $\mu_1$-measure. It is enough to show that for any $x_1\in \pi_1(\Omega_\uu \cap \Gamma_1)$ there exists exactly one $\xx=(x_1,x_2,\dots,x_m)$ which belongs to $\Omega_\uu \cap \Gamma_1$.

Fix $x_1\in \pi_1(\Omega_\uu \cap \Gamma_1)$, and let $\xx=(x_1,\dots,x_m)\in \Omega_\uu \cap \Gamma_1$. Letting $y$ be a barycenter for $\xx$, we have that $y \in \partial_{\dc^2}u_1(x_1)$ by \cref{lemma: dconcave}. Then \cite[Lemma 4.8]{Ambrosio2004} implies that $\partial_{\dc^2}u_1(x_1)$ is a singleton, so that $y$ is the \emph{unique} barycenter of $\xx$; here we used the differentiability of $u_1$ along each direction $X_j$, $Y_j$, $Z$. By our assumption \eqref{equation:condition_C1}, then, $\xx$ is uniquely determined by $\mathbf{b}_{\Gamma_1}^{-1}(y)$: this proves that for any $x_1\in \pi_1(\Omega_\uu \cap \Gamma_1)$ there exists exactly one such $\xx$. Moreover, by \cite[Lemma 4.8 and Theorem 5.1]{Ambrosio2004}, the unique barycenter $y$ can be expressed as
	\begin{equation}
	x_1 \cdot \exp_\heis (-X u_1 (x_1)-iY u_1(x_1), -Z (x_1)),
	\end{equation}
which proves our last statement.

Uniqueness of the optimal transport plan $\gamma$ and map $\Psi$ then follows by a standard argument; if $\gamma_0$ and $\gamma_1$ are both optimal, then so is $\gamma_{1/2}:=\frac{1}{2}(\gamma_0+\gamma_1)$, since \eqref{equation: general kantorovich} is a linear minimization over a convex set.  The argument above implies that $\gamma_i$ is induced by a transport map $\Psi_i$, for $i=0,1$; $\gamma_{1/2}$ is then concentrated on the union of these two graphs.  Another application of the argument above means that $\gamma_{1/2}$ must also be concentrated on a graph; this is a contradiction unless $\Psi_0=\Psi_1$ almost everywhere, in which case $\gamma_0=\gamma_1$.
\end{proof}

\begin{remark}
By tracing back our argument, one can see that only the absolute continuity of $\mu_1$ is actually required; thus \cref{theorem: monge} could be stated under this weaker assumption. However, we will soon need the absolute continuity of $\mu_2,\dots,\mu_m$ as well: for this reason, we decided to introduce this assumption already at this stage.
\end{remark}

The real nature of the maps $\psi$ and $\mathbf{f}$ appearing in \cref{theorem: monge} is clarified by the following \namecref{prop: optimal maps}. Recall that the \emph{Wasserstein barycenter} of the measures $\mu_i$ a probability measure which minimizes the functional	
	\begin{equation}\label{eqn: barycenter functional}
	\nu \mapsto \sum_{i=1}^mW_2^2(\mu_i,\nu)
	\end{equation}
among all probability measures $\nu$ on $\heis^n$, where $W_2^2(\mu_i,\nu)$ is the squared Wasserstein distance between $\mu_i$ and $\nu$:
	\begin{equation}\label{eqn: wasserstein distance}
	W_2^2(\mu_i,\nu):=	\inf \brcs*{\int_{(\heis^n)^2} d_c^2(x_i,y) \,d\gamma(x_1,y)\stset \gamma \in\Pi\pths*{\mu_1,\nu}}.
	\end{equation}
Existence of a Wasserstein barycenter follows immediately by basic continui\-ty-compactness arguments.  Absolute continuity of $\mu_1$, together with Theorem 5.1 in \cite{Ambrosio2004} imply via a standard argument (see Proposition 7.19 in \cite{Santambrogio2015}) that the functional \eqref{eqn: barycenter functional} is strictly convex, and so the Wasserstein barycenter is unique.  The next result asserts an equivalence between the Wasserstein barycenter and multi-marginal optimal transport problems, analogous to the relationship on $\mathbb{R}^n$ due to \cite{AguehCarlier2010} and its extension to Riemannian manifolds established in \cite{KimPass2017}. Stated simply, each optimal component mapping in \eqref{equation: general Monge} pushing $\mu_1$ forward to $\mu_i$ is the composition of the two marginal optimal mapping for the quadratic cost $d_c^2$ between $\mu_1$ and the Wasserstein barycenter $\nu$, and the (two marginal quadratic cost) optimal mapping from $\nu$ to $\mu_i$.

\begin{proposition} \label{prop: optimal maps}
Let the assumptions and the notations be the same as in \cref{theorem: monge}. Let $\nu$ be the Wasserstein barycenter of the measures $\mu_i$ (\ie{}, $\nu$ minimizes \eqref{eqn: barycenter functional}). Then:
	\begin{enumerate}[(i)]
	\item The map $\psi: \heis^n \to \heis^n$ in \eqref{eqn: optimal maps structure}  is the optimal map carrying $\mu_1$ to $\nu$;
	\item For each $i=2,\dots,m$, the map $\mathbf{f}_i: \mathbf{b}(\Gamma_1)\to \heis^n$ in \eqref{eqn: optimal maps structure} is the optimal map carrying $\nu$ to $\mu_i$.
	\end{enumerate}
\end{proposition}

\begin{proof}
By \cref{lemma: dconcave}, $u_1$ is a $\dc^2$ concave map; moreover, by assumption, it is differentiable along the directions $X_j$, $Y_j$, $Z$ for $\mu$-a.e.~$x_1\in\heis^n$. The last part of \cite[Theorem 5.1]{Ambrosio2004} then implies that $\psi$ is the optimal transport map between $\mu_1$ and $\psi_\sharp \mu_1$: thus we only need to show that $\psi_\sharp \mu_1=\nu$. On the other hand, we can represent $\psi$ as $\mathbf{b}_{\Gamma_1}\circ (\Id, \Psi)$; thus 
	\begin{equation}
	\psi_\sharp \mu_1= (\mathbf{b}_{\Gamma_1})_\sharp \gamma,
	\end{equation}
where $\gamma$ is the optimal plan induced by $\Psi$; by a result of Carlier and Ekeland \cite{Carlier2010}, the push-forward of the optimal plan through the barycenter map is the Wasserstein barycenter itself, thus proving our statement.  

Now, assumption \eqref{equation:condition_C1} implies that $\mathbf{b}_{\Gamma_1}$ is invertible almost everywhere.  Since this mapping is optimal between $\mu_1$ and the Wasserstein barycenter $\nu$, its inverse $\mathbf{f}_1$ is the optimal map from $\nu$ to $\mu_1$.  An identical argument implies that each $f_i$ is optimal between $\nu$ and $\mu_i$.
\end{proof}

We next identify a simple condition on the measures $\mu_1,\dots,\mu_m$ which guarantees that the (rather complicated) assumption \eqref{equation:condition_C1} on $\Gamma_\uu$ is satisfied. Unfortunately, the condition we impose is quite strict and leaves out some interesting cases.  Though the condition is needed for technical reasons here, we suspect it  can at least be weakened in some way.

The following \namecref{lemma: same bc} encodes the following geometric fact: let $\xx$ be a $m$-tuple of points in $(\heis^n)^m$; under a suitable non-verticality assumption, if we move one of the points $x_i$ and leave the others fixed, then the barycenters also move.

\begin{lemma}\label{lemma: same bc}
Let $\xx=(x_1,x_2,\dots,x_m)$ and $\xx'=(\bar{x}_1,x_2,\dots,x_m)$ be two $m$-tuples having a common barycenter $y\in\heis^n$ with respect to the CC-distance $\dc$. If $x_i\notin y\cdot L$ for all $i\geq 2$, then $x_1=\bar{x}_1$.
\end{lemma}
Here we denote by $L$ the vertical axis $\brcs*{[0,s]\in\heis^n\stset s\in\R}$.

\begin{proof}
Let us define, for $z\in\heis^n$: 
	\begin{align}
	\phi_i(z)&\doteq \dc^2(x_i,z) \qquad \forall i=2,\dots,m\\
	\Phi(z)&=-\sum_{i=2}^m\phi_i(z).
	\end{align}
Since $y$ is a barycenter for $\xx$ and $\xx'$, by definition of $\dc^2$-superdifferential we have that both $x_1$ and $\bar{x}_1$ belong to $\partial_{\dc^2}\Phi(y)$. But now each $\phi_i$ is differentiable at $y$ (by \cite[Lemma 3.11]{Ambrosio2004} and by the fact that $x_i\notin y\cdot L$), thus $\Phi$ is also differentiable at $y$. By \cite[Lemma 4.8]{Ambrosio2004}, this implies that $\partial_{\dc^2}\Phi(y)$ is a singleton.
\end{proof}

Finally, we introduce a second assumption on the mesures $\mu_1,\dots,\mu_1$; one can think of it as a request for the supports of the measures to be ``sufficiently far'' from one another:
\begin{definition}
Let $\mu_1,\dots,\mu_m$ be compactly supported, absolutely continuous probability measures on $\heis^n$. We say that $\mu_1,\dots,\mu_m$ satisfy the condition \eqref{equation:condition_C2} if:
	\begin{equation}\label{equation:condition_C2}\tag{C2}
	\parbox{.8\textwidth}{\centering \itshape
	the barycenters set $\mathbf{B}\doteq \mathbf{b}\pths*{\prod_i \spt(\mu_i)}$ has zero $\mu_i$-measure for all $i=1,\dots,m$.
	}
	\end{equation}
Here 
	\begin{equation}
	\mathbf{B}=\mathbf{b}\pths*{\prod_{i=1}^m \spt(\mu_i)}\doteq \brcs*{y\in\heis^n \stset 
		\begin{gathered}
		\text{$y$ is a barycenter for $(x_1,\dots,x_m)$,}\\
		\text{with $x_i\in\spt (\mu_i)$ for all $i$}
		\end{gathered}			
	}.
	\end{equation}
In this case, we denote by $F_i\subset \heis^n$ the $\mu_i$-zero measure set
	\begin{equation}
	F_i \doteq \mathbf{B}\cap \spt(\mu_i)
	\end{equation}
and by $\Omega_{\mathrm{bc}}\subset (\heis^n)^m$ the set
	\begin{equation}\label{equation: omega_bc}
	\Omega_{\mathrm{bc}}\doteq \prod_{i=1}^m (\heis^n\setminus F_i),
	\end{equation}
which has full $\mu$-measure for any transport plan $\mu\in\Pi(\mu_1,\dots,\mu_m)$.
\end{definition}

\begin{lemma}\label{lemma: bc determines xx}
Let $\mu_1,\dots,\mu_m$ be compactly supported, absolutely continuous probability measures on $\heis^n$, and let $c$ be the cost associated to $\dc$. Assume that the $\mu_i$'s satisfy condition \eqref{equation:condition_C2}. Let $\uu=(u_1,\dots,u_m)$ be a $c$-conjugate solution to the dual problem~\eqref{equation: dual problem}. Let $\xx, \bar{\xx}\in \Omega_\uu\cap \Omega_{\mathrm{bc}}\cap \Gamma_\uu$, where $\Omega_\uu$ is the set defined in \cref{equation: Omega} and $\Omega_{\mathrm{bc}}$ is the set defined in \cref{equation: omega_bc}. If $y\in\heis^n$ is a barycenter for both $\xx$ and $\bar{\xx}$, then $\xx=\bar{\xx}$.
\end{lemma}

\begin{proof}
We show that, for example, $x_1=\bar{x}_1$ (the same argument then applies to $x_2,\dots,x_m$). As a first observation, notice that by \cref{lemma: swap} $y$ is also a barycenter for $\xx'=(\bar{x}_1,x_2,\dots,x_m)$. Furthermore, by our assumption \eqref{equation:condition_C2} and our definition of $\xx$, we have that $y\neq x_i$ for any $i$. By \cref{lemma: same bc}, we reach the conclusion if we can also show that $x_i\notin y\cdot L^\ast $ for all $i\geq 2$, where $L^\ast=L\setminus\brcs{0}$. By \cref{lemma: dconcave}, $u_i$ is $\dc^2$-concave for each $i$, with $y\in\partial_{\dc^2}u_i(x_i)$. Since we assumed $u_i$ to be (horizontally) differentiable at $x_i$, then we can apply \cite[Lemma 4.7]{Ambrosio2004}, which ensures $\partial_{\dc^2}u_i(x_i)\cap (x_i\cdot L^\ast)=\emptyset$.
\end{proof}

As a consequence, the following result can be stated:
\begin{corollary}\label{corollary: final result}
Let $\mu_1,\dots,\mu_m$ be compactly supported, absolutely continuous probability measures on $\heis^n$, and let $c$ be the cost associated to $\dc$. Assume that the $\mu_i$'s satisfy condition \eqref{equation:condition_C2}. Let $\uu=(u_1,\dots,u_m)$ be a $c$-conjugate solution to the dual problem~\eqref{equation: dual problem}. Then the set $\Gamma_\uu$ satisfies condition \eqref{equation:condition_C1}; in particular, \cref{theorem: monge} holds in this case. 
\end{corollary}

\begin{remark}[Gauge distance]
Up to a slight modification in the definition of the set $\Omega_\uu$ (\cref{equation: Omega}), the proof of \cref{theorem: monge} is still valid if we consider the Gauge distance $\dg$ on $\heis^n$ instead of $\dc$:
	\begin{equation}
	\dg([\zeta, t], [\zeta', t'])\doteq \sqrt[4]{\abs*{\zeta - \zeta'}^4+ (t-t')^2}.
	\end{equation}
Indeed, using the same notations and assumptions of \cref{theorem: monge}, one still needs to exclude the points $x_1\in \heis^n$ at which $u_1$ is not differentiable, but also the points at which $X_ju_1= Y_ju_1=0$ for every $j$ and $Zu_1\neq 0$: thanks to \cite[Lemma 4.13]{Ambrosio2004}, the set we are excluding is still $\leb^{2n+1}$-negligible; moreover, thanks to \cite[Lemma 4.11]{Ambrosio2004}, this new extended assumption again ensures that $\partial_{\dg^2} u_1 (x_1)$ is a singleton; the rest of the proof follows in the same way.
\end{remark}

\begin{remark}[Open problem: more general Carnot groups]
It is possible that the techniques we exploited can be generalized to include a wider class of Carnot groups. Specifically, it seems a promising idea to adapt our argument to the case of groups of type $H$, making use of the two-marginals results contained in \cite{Rigot2005}.
\end{remark}

\begin{remark}[Open problem: regularity of Wasserstein barycenters]
	As in \cite{AguehCarlier2010} and \cite{KimPass2017} (and as we already saw in \cref{prop: optimal maps}), under the conditions in Theorem \ref{theorem: monge}, the measure
	\begin{equation}
\nu:=\Big(x_1\mapsto	x_1 \cdot \exp_\heis (-X u_1 (x_1)-iY u_1(x_1), -Z u_1(x_1))\Big)_\#\mu_1
	\end{equation}
	is the Wasserstein barycenter of the $\mu_i$ (with equal weights); 
 When $m=2$, a result of \cite{FigalliJuillet2008} establishes that the barycenter is in fact absolutely continuous with respect to Lebesgue measure.  The strategy of proof there has been used to establish absolute continuity of Wasserstein barycenters on Riemannian manifolds for $m \geq 3$ in \cite{KimPass2017}.  It relies in a crucial way on the Measure Contraction Property (or, in \cite{KimPass2017} on a barycentric version of it).  This property (for two marginals) was established on the Heisenberg group in \cite{Juillet2009}; whether a barycentric version of it holds, as it does on Riemannian manifolds, is an interesting open question.  If so, it would presumably allow one to prove absolute continuity of the Wasserstein barycenter for $m \geq 3$.   
\end{remark}
\subsection{Extension to higher power costs.}\label{subsection: higher p}
Ambrosio-Rigot noted that their two marginal results extend to the case when the cost function $c=d^2$ is replaced by $d^p$ for $p \geq 2$.  Similarly, in our multi-marginal case
 if we choose an exponent $p\geq 2$ in the definition of the cost function, \ie{}
	\begin{equation}
	c_{d;p}(x_1,\dots,x_m)\doteq \inf_{y\in\heis^n}\brcs*{ \sum_{i=1}^m d^p(x_i, y) },
	\end{equation}
	then everything we proved still works. Moreover, if $p>2$, the function $y \mapsto d^p(x_i,y)$ is differentiable as long as $x_i\notin y\cdot L^\ast$ (that is, unlike in the $p=2$ case, differentiability holds at $y=x_i$); this is because $d^p([0,0],[0,t]) =\pi|t|^{p/2}$, which is differentiable at $t=0$ for $p>2$ (the distance itself is differentiable in the horizontal directions).  Therefore, $\Phi$ defined in \cref{lemma: same bc} is differentiable at $y$ even if $x_i=y$ for some $i$: 	this means that the \namecref{lemma: same bc} holds true with the weaker assumption $x_i\notin y\cdot L^\ast$. In particular, in this modified setting, \cref{lemma: bc determines xx} and \cref{corollary: final result} hold \emph{without} the need of assumption \eqref{equation:condition_C2}.

In this case, our multi-marginal optimal maps in \eqref{equation: general Monge} are compositions of the two marginal optimal maps characterized in section 7 of \cite{Ambrosio2004}; the first factor pushes $\mu_1$ forward to the $p$-th barycenter of the $\mu_i$ (that is, the $\nu$ which minimizes $\nu \mapsto \sum_{i=1}^mW_p^p(\mu_i,\nu)$, where $W_p$ is the $p$-Wasserstein distance); the second pushes this $\nu$ forward to $\mu_i$.

\nocite{Ambrosio2013,Ambrosio2004,Villani2009}
\emergencystretch=1em
\printbibliography
\end{document}